\documentclass[a4paper,reqno, oneside,11pt]{amsart}
\usepackage[english]{babel}
\usepackage{amsmath, amssymb, amsthm, amscd}
\usepackage{enumerate}
\usepackage{palatino}
\usepackage{mathpazo}
\usepackage{paralist}
\usepackage[a4paper]{geometry}

\usepackage{tikz-cd}

\newcommand{\Z}{{\mathbb {Z}}} 
\newcommand{\R}{{\mathbb {R}}} 
\newcommand{\C}{{\mathbb {C}}} 
\newcommand{\Q}{{\mathbb {Q}}} 
\newcommand{\cat}{{\sf {cat}}} 
\newcommand{\secat}{{\sf {secat}}} 
\newcommand{\tc}{{\sf {TC}}} 
\newcommand{\zcl}{{\sf {zcl}}} 

\theoremstyle{plain}
\newtheorem{theorem}{Theorem}[section]
\newtheorem{proposition}[theorem]{Proposition}
\newtheorem{lemma}[theorem]{Lemma}
\newtheorem{cor}[theorem]{Corollary}  

\theoremstyle{definition}
\newtheorem{definition}[theorem]{Definition}

\newtheorem{question}[theorem]{Question}

\theoremstyle{remark}
\newtheorem{remark}[theorem]{Remark}
\newtheorem{example}[theorem]{Example}

\numberwithin{equation}{section}
\thanks{The author thanks Michael Farber and Jes\'{u}s Gonz\'{a}lez for the invitation to contribute this survey to a forthcoming book project  as well as Arturo Espinosa Baro, Maximilian Stegemeyer and the anonymous reviewer for helpful comments on an earlier draft of the manuscript.}
\begin{document}
\title[Geometric and topological properties of manifolds in robot motion planning]{Geometric and topological properties of manifolds in robot motion planning}
\author{Stephan Mescher}
\address{Institut f\"ur Mathematik \\ Martin-Luther-Universit\"at Halle-Wittenberg \\ Theodor-Lieser-Strasse 5 \\ 06120 Halle (Saale) \\ Germany}
\email{stephan.mescher@mathematik.uni-halle.de}
\maketitle

\begin{abstract}
	Manifolds occur naturally as configuration spaces of robotic systems. They provide global descriptions of local coordinate systems that are common tools in expressing positions of robots. The purpose of this survey is threefold. Firstly, we present an overview over various results on topological complexities of manifolds and related topics. Several constructions for manifolds, e.g. symplectic structures and connected sums, can be used to compute or estimate topological complexities. Secondly, we take a look at geodesic motion planning in Riemannian manifolds. In this setting, results from Riemannian geometry are employed to estimate the complexity of motion planning along shortest paths in manifolds. Thirdly, we will discuss results on connections between critical point theory and the topology of manifolds that are related to motion planning problems. Here, we consider the role of navigation functions for topological complexity and outline their relations to newer numerical homotopy invariants, namely spherical complexities. \end{abstract}



\section*{Introduction}

In the mathematical treatment of robotics, the first important task is to find a mathematical representation of the configuration space of the robotic system. This task usually amounts to finding a suitable set of generalized coordinates, i.e. real-valued parameters, which uniquely determine the state of the system. To find the minimal number of coordinates required by the state of the system, one determines the number of degrees of freedom, call it $k$, of the mechanic system via Gr\"ubler's formula, see \cite[Proposition 2.2.2]{LynchPark}. Then one establishes $k$ real-valued coordinates that determine the positions of the links of the system at least locally, e.g. the length of a prismatic joint, the angle of rotation of a crank or other observable quantities. An elementary overview of this approach is given in \cite[Chapter 2]{LynchPark}. In many situations, the local coordinates can be assembled in such a way that the configuration space or at least a large part of it forms a smooth manifold. A basic example is given by a planar linkage consisting of $n$ links such that each is connected to the next one by a crank. Assuming there are no obstructions and that all bars can be rotated freely, the configuration space of this system is given by the $n$-torus $T^n=(S^1)^n$, which is a closed $n$-dimensional manifold. Here, we represent the linkage by the $n$ rotational angles of the cranks. In general configuration spaces will have singularities, which usually occur in the form of critical points of the maps which encode the constraints of the system. One might still work one's way around these singularities by just considering the \emph{regular} configuration space, i.e. the complement of the singularities, and make use of the manifold structure of the complement. See e.g. \cite{Pfalzgraf} for a discussion of related questions or \cite{HaugPavesic} for a recent approach in this direction.

Surprisingly, it has been shown by M. Kapovich and J. Millson that \emph{every} manifold can occur in the study of topological robotics: in \cite[Corollary C]{KapoMillson} it is shown that for \emph{any} closed smooth manifold $M$ there exists a planar linkage whose configuration space is given by a disjoint union of copies of $M$. This emphasizes the connections between topological robotics and the geometry and topology of manifolds, which gives reasons to study connections between the two fields.\medskip 

Motivated by these observations, we want to present topics of current research in the geometry and topology of motion planning in smooth manifolds. Our presentation is divided into three sections, which can be studied mostly independently from one another. Simultaneously, we want to give an overview of the literature and of further results that are not elaborated upon here. A reader who is not too familiar with smooth manifolds is recommended to consider \cite{LeeSmooth} or a comparable textbook to look up the basic facts. We further assume familiarity with basic results on topological complexity.\medskip 

In the first section of this article, we survey results on the topological complexity of manifolds, focussing on those cases in which explicit use of their topological properties is made. After considering frame bundles as models for work spaces in robotics, we discuss how cohomology rings of manifolds can be employed to obtain lower bounds on topological complexity. Here, results like Poincar\'{e} duality that are peculiar to manifolds come in handy in various ways. In Section 2 we discuss how one can model motion planning tasks along paths of shortest length in Riemannian manifolds. More precisely, we outline some constructions and results on geodesic complexity, which is defined in the more general context of metric spaces. In the case of a Riemannian manifold the behaviour of geodesics can to a large extent be observed from its tangent bundle - another notion that is peculiar to smooth manifolds. Afterwards, we study connections between critical points of differentiable functions and manifold topology in the context of topological motion planning in Section 3. We first elaborate upon navigation functions, which have been introduced by M. Farber as a method to obtain upper bounds on the complexity of motion planning in a closed manifold and to construct explicit motion planners from gradient flows. Afterwards, we discuss how a similar approach can be used in Riemannian geometry to tackle the closed geodesics problem via so-called spherical complexities that have recently been introduced by the author.

\section{Topological complexity of manifolds}

 
An elementary upper bound on the topological complexity (TC) of a manifold has already been shown by M. Farber in his seminal article \cite{FarberTC}. The TC of an $r$-connected manifold $M$, where $r \in \mathbb{N}_0$, satisfies 
\begin{equation}
\label{EqUpperTCdimconn}
\tc(M) \leq \frac{2\dim M}{r+1}.
\end{equation}
 Most \emph{lower} bounds on topological complexity, or more generally on sectional categories of fibrations, are derived from properties of cohomology rings of the spaces involved. A well-known lower bound for TC is the one by the zero-divisor cup length of the space, see \cite[Theorem 7]{FarberTC} or \cite[Corollary 4.40]{FarberBook}.
 
We will discuss in this section how these established bounds can be improved in one way or another if the space is a smooth manifold, where in several situations, topological properties of manifolds like Poincar\'{e} duality will come in handy.

\subsection{Lower bounds via cohomology} Before considering more elaborate constructions, we want to show the reader an explicit example of how cohomological properties can be used to estimate TC from below. For this purpose, we include the following computation of the TC of certain homogeneous manifolds, which slightly generalizes the one of the topological complexity of simply connected closed symplectic manifolds from \cite[Lemma 28.1]{FarberSurveyTC}. To the author's best knowledge, its assertion is new to the literature.

\begin{theorem} 
\label{ThmTCGT}
Let $G$ be a compact connected Lie group and let $T\subset G$ be a maximal torus. Then 
$$\tc(G/T) =\dim G - \mathrm{rk}\, G \ , $$
where $\mathrm{rk}\, G=\dim T $ denotes the rank of $G$, i.e. the dimension of any maximal torus of $G$.
\end{theorem}
\begin{proof}
As shown e.g. in \cite[Proposition 13.36]{HallLie}, the quotient space $G/T$ is simply connected. Thus, the upper bound \eqref{EqUpperTCdimconn} yields
\begin{equation}
\label{TCleq}
\tc(G/T) \leq \dim (G/T)= \dim G - \mathrm{rk}\, G .
\end{equation}
To obtain the same number as a lower bound on $\tc(G/T)$, we observe from \cite[p. 114]{SinghofLieI} that 
\begin{equation}
\label{clineq}
{\sf {cl}}_{\R}(G/T) \geq \frac12 \dim (G/T)\ ,
\end{equation} 
where ${\sf {cl}}_{\R}$ denotes the cup length of $G/T$ with real coefficients. By a classical theorem of Borel, see \cite[Lemma 26.1]{Borel}, the cohomology groups $H^*(G/T;\R)$ are concentrated in even degrees.

Put $n:= {\sf {cl}}_{\R}(G/T)$ and let $u_1,\dots,u_n \in \widetilde{H}^*(G/T;\R)$ be given with $$u_1\cup u_2\cup \dots \cup u_n \neq 0.$$ For each $i \in \{1,2,\dots,n\}$ we let $\bar{u}_i \in H^*(G/T\times G/T;\R)$, 
$$\bar{u}_i := 1 \times u_i - u_i \times 1,$$
be its associated zero-divisor. Since $u_i$ is of even degree, it holds that 
$$\bar{u}_i^2 = 1 \times u_i^2 - 2 u_i \times u_i + u_i^2 \times 1. $$
From this computation, we observe that 
$$\bar{u}_1^2 \cup \dots \cup \bar{u}_n^2 = (-2)^n (-1)^n (u_1 \cup u_2 \cup \dots \cup u_n) \times (u_1 \cup u_2 \cup \dots \cup u_n) \neq 0, $$
since all other summands that occur in the product will have cup products of at least $n+1$ elements of $\widetilde{H}^*(G/T;\R)$ in one of its factors. We derive that
$$\zcl_{\R}(G/T) \geq 2n \stackrel{\eqref{clineq}}{\geq} \dim (G/T) \quad \Rightarrow \quad \tc(G/T) \geq \dim G/T.$$
Combining this observation with \eqref{TCleq} shows that
$$\tc(G/T)=  \dim (G/T)= \dim G - \dim T = \dim G - \mathrm{rk}\, G \ . $$
\end{proof}

The proof of the previous theorem shows how knowledge about the cohomology ring of a space  can be employed to derive lower bounds on topological complexity. However, there are more sophisticated ways of making use of cohomological properties, in particular if the space under investigation is a manifold. 

In \cite{CostaFarber}, A. Costa and M. Farber have singled out a certain cohomology class of cell complexes, henceforth called \emph{canonical class},  which plays a particularly important role in the study of TC. Given a cell complex $X$ with fundamental group by $\pi:=\pi_1(X)$, it is given as a class $\mathfrak{v}_X \in H^1(X\times X;I)$, where $I$ denotes the augmentation ideal in the integer group ring $\Z[\pi]$, seen as a left $\Z[\pi\times\pi]$-module via the $(\pi \times \pi)$-action given by $(g,h)\cdot x= gxh^{-1}$. 

This class can be used to obtain lower bounds for TC. Among other results on canonical classes, Costa and Farber have shown in \cite[Theorem 7]{CostaFarber} that for an $n$-dimensional finite cell complex, it holds that
\begin{equation}
\label{EqTCmaxequiv}
\tc(X)=2n \quad \Leftrightarrow \quad \mathfrak{v}^{2n}_X \neq 0 \in H^{2n}(X\times X;I^{\otimes 2n}),
\end{equation}
where $\mathfrak{v}^{2n}_X$ denotes the $2n$-fold cup product of $\mathfrak{v}_X$ with itself.

 It is known from classical Morse theory, see \cite{MilMorse}, that each closed manifold has the homotopy type of a finite cell complex. Thus, in the case that $X$ is a closed manifold, we can apply \eqref{EqTCmaxequiv} and derive a more explicit criterion on the non-vanishing of $\mathfrak{v}^{2n}_X$ by means of Poincar\'{e} duality. 
 \begin{theorem}
\label{Theoremvv2n}
Let $M$ be a closed $n$-dimensional manifold, where $n \in \mathbb{N}$. Then $\tc(M)=2n$ if and only if 
$$ \mathfrak{v}_M^{2n} \cap [M\times M] \neq 0 \in H_0(M\times M;I^{\otimes n}\otimes o_{M \times M}),$$
where $o_{M \times M}$ denotes the orientation bundle of $M \times M$.
 \end{theorem}

Thus, the question if the TC of a manifold is maximal reduces to the question if a certain degree-zero homology class vanishes or not. This observation was used by D. Cohen and L. Vandembroucq in several ways and the interested reader is referred to L. Vandembroucq's contribution to this volume, in which some ingenious computations are carried out. Here, we only want to mention that Cohen and Vandembroucq have computed the TC of the Klein bottle $K$ in \cite{CVKlein} as $\tc(K)=2 \dim K = 4$ and have shown in \cite{CVabelian} that $\tc(M)<2 \dim M$ in various cases of manifolds with abelian fundamental groups.  \medskip

We want to focus on a different cohomological technique for lower bounds on TC. We recall that TC is obtained as a special case of the notion of \emph{sectional category of a fibration}. Here, the sectional category of a fibration $p:E \to B$, denoted by $\secat(p\colon E\to B)$, or simply $\secat(p)$, is the minimal value of $r \in \mathbb{N}$ for which $B$ admits an open cover $U_0,U_1,\dots,U_r$, such that $p$ admits a continuous local section over each of the $U_i$. For any path-connected space $X$, we let $PX=C^0([0,1],X)$ be its path space and consider the path fibration $$\pi: PX \to X \times X, \qquad \pi(\gamma)=(\gamma(0),\gamma(1)).$$ 
Then  
$$\tc(X)=\secat(\pi\colon PX \to X \times X).$$
In \cite{FarberGrantSymm}, M. Farber and M. Grant have introduced the notion of \emph{sectional category weight} of a cohomology class of a fibration. It generalizes notions of category weight introduced by E. Fadell and S. Husseini in \cite{FadellHusseini} and by Y. Rudyak in \cite{RudyakWeight}. We give a definition that is slightly different from the one of Farber and Grant, but which is shown in \cite{FarberGrantSymm} to coincide with the original definition. We recall that the fiberwise join of two fibrations $p_1\colon E_1 \to B$ and $p_2\colon E_2 \to B$ over the same base space is another fibration $p_1 *_f p_2: E_1 *_f E_2 \to B$, such that the fiber of $p_1 *_f p_2$ has the homotopy type of the join of the fibers of $p_1$ and $p_2$, see e.g. \cite[p. 11]{CrabbJames}. 
\begin{definition}
Given a fibration $p\colon E \to B$, we let $p_k\colon E_k \to B$, denote the $k$-fold fiberwise join of $p$ with itself for each $k \in \mathbb{N}$. Given $u \in \widetilde{H}^*(B;A)$ with $u \neq 0$, where $A$ is an abelian group, the \emph{sectional category weight of $u$ with respect to $p$} is given by 
$$\mathrm{wgt}_p(u):= \sup \{ k \in \mathbb{N} \ | \ p_k^*(u) =0\}. $$
\end{definition}
The following theorem summarizes the most important properties of sectional category weights. For simplicity, we restrict the result to coefficient groups and only want to mention that the whole setting of sectional category weights extends to local coefficient systems in a straightforward way. Its assertions are shown in \cite[Section 6]{FarberGrantSymm}.
\begin{theorem}
\label{TheoremWeights}
Let $p\colon E \to B$ be a fibration, let $A$ be an abelian group and let $u \in \widetilde{H}^*(B;A)$ with $u \neq 0$.
\begin{enumerate}[a)]
	\item If $\mathrm{wgt}_p(u)=k$, then $\secat(p:E \to B)\geq k$.
	\item If $f\colon X\to B$ is continuous and $f^*u \neq 0$, then $\mathrm{wgt}_{f^*p}(f^*u)\geq \mathrm{wgt}_p(u)$. 
	\item If $R$ is a commutative ring and $u_1,u_2\dots,u_r \in \widetilde{H}^*(B;R)$ satisfy $u_1 \cup u_2 \cup \dots \cup u_r \neq 0$, then
$$\mathrm{wgt}_p(u_1 \cup u_2 \cup \dots \cup u_r) \geq \sum_{i=1}^r \mathrm{wgt}_p(u_i). $$
\end{enumerate}
\end{theorem}

In the case of topological complexity, i.e. of $p=\pi:PX\to X \times X $, we shall employ the notation 
$$\mathrm{wgt}_{\tc}(u):= \mathrm{wgt}_{\pi}(u)$$
and call it \emph{the TC-weight of $u$}. 

There are various techniques for finding classes whose TC-weight exceeds one. Here, we want to discuss a method that was first suggested by Mark Grant and used by Grant and the author in \cite{GrantMescher} to compute the topological complexity of certain symplectic manifolds as we shall elaborate upon below, see also \cite[Section 4.2]{MescherSC}.

Consider again a fibration $p\colon E \to B$. Cohomology classes whose weight with respect to $p$ is at least two lie in the kernel of $p_2^*$, where $p_2\colon E_2 \to B$ denotes the fiberwise join of $p$ with itself. In terms of homotopy theory, this fiberwise join is  given as the homotopy pushout of the pullback $Q$ of $E \stackrel{p}{\rightarrow} B \stackrel{p}{\leftarrow} E$. As such, its cohomology with coefficients in an abelian group $A$ admits a Mayer-Vietoris sequence of the form
\begin{equation*}
	\dots \longrightarrow H^{k-1}(Q;A) \stackrel{\delta}{\longrightarrow} H^k(E_2;A) \stackrel{i_1^*\oplus i_2^*}{\longrightarrow} H^k(E;A)\oplus H^k(E;A) \longrightarrow H^k(Q;A) \longrightarrow \dots 
\end{equation*}
Here, $i_1$ and $i_2$ denote the whisker maps of the homotopy pushout diagram, see \cite[Section 2]{GrantMescher} for details.
In the case of the path fibration $\pi:PX \to X \times X$, it is easy to see that $Q \approx LX$, the free loop space $LX=C^0(S^1,X)$ of $X$, where each loop is imagined as the concatenation of two paths with the same endpoints. Moreover, the following diagram, whose upper row is the above Mayer-Vietoris sequence, commutes
\begin{equation*}
{\begin{tikzcd}
	\dots \arrow[r] & H^{k-1}(LX;A) \arrow[r, "\delta"] &  H^k(E_2;A) \arrow[r, "i_1^*\oplus i_2^*"]& H^k(PX;A)\oplus H^k(PX;A) \arrow[r] & \dots \\
	& & H^k(X\times X;A)\arrow[u,"p_2^*"] \arrow[ur, "\pi^* \oplus \pi^*"']
	\end{tikzcd}}
\end{equation*}
Consider a class $\sigma \in H^k(X \times X;A)$. Then $\sigma$ is a zero-divisor iff $\sigma \in \ker \pi^*$. Using this property and the exactness of the top row in the above diagram, we derive that 
$$p_2^*\sigma \in \mathrm{im}\left[\delta:H^{k-1}(LX;A) \to H^k(P_2X;A)\right].$$ 
Thus, a method of showing that $p_2^*\sigma=0$, i.e. that $\mathrm{wgt}_{\tc}(\sigma)\geq 2$, would be to 
\begin{itemize}
	\item Find a class $\alpha_\sigma \in H^{k-1}(LX;A)$ with $\delta(\alpha_\sigma) = p_2^*\sigma$. 
	\item If $\alpha_\sigma=0$, then $p_2^*\sigma= \delta(0)=0$, hence $\mathrm{wgt}_{\tc}(\sigma) \geq 2$.
\end{itemize}

In \cite[Section 3]{GrantMescher} and \cite[Section 4]{MescherSC}, we discuss how to find such a class $\alpha_\sigma$ if $\sigma$ is of the form 
$$\sigma = \bar{u} := 1 \times u - u \times 1$$ 
for some $u \in H^k(X;A)$. Here, $\times$ denotes the cohomology cross product and a simple computation shows that each class of the form $\bar{u}$ is indeed a zero-divisor. One obtains that the class $\alpha_\sigma$ is for $\sigma = \bar{u}$ given by 
$$\alpha_\sigma= \mathrm{ev}^*u/[S^1], $$
 where $\mathrm{ev}\colon LX \times S^1 \to X$, $\mathrm{ev}(\alpha,t)=\alpha(t)$, where $[S^1]$ is the fundamental homology class of $S^1$ and where $\cdot / \cdot$ denotes the slant product, see \cite[Section 6.1]{Spanier}. 

The final ingredient to derive a vanishing criterion for this class is a celebrated result by R. Thom who has shown in \cite{ThomQuelques} that every degree-$k$ \emph{rational} homology class is obtained as a pushforward of a fundamental class of a suitable $k$-dimensional closed manifold. Using the above description, Thom's result and Kronecker duality, one establishes the following result.
 
\begin{theorem}[{\cite[Proposition 5.3]{MescherSC}}]
\label{TheoremPullbackWeight}
	Let $X$ be a path-connected Hausdorff space, let $k \geq 2$ and let $u \in H^k(X;\Q)$. If $f^*u=0$ for all continuous maps of the form $f\colon P \times S^1 \to X$, where $P$ is a $k$-dimensional closed oriented manifold, then 
	$$\mathrm{wgt}_{\tc}(\bar{u})\geq 2.$$
\end{theorem}
\begin{remark}
In the manifold case, the pullback condition occurring in Theorem \ref{TheoremPullbackWeight} can be derived from other, more geometric properties. More precisely, if we assume that $X$ is a smooth manifold, then the hypothesis of Theorem \ref{TheoremPullbackWeight} will be satisfied if one of the following holds:
\begin{itemize}
	\item  $X$ admits a metric of negative sectional curvature, see \cite[Theorem 3.4]{Neofyt},
	\item $u$ lies in the image of the comparison map $H^k_b(X;\R) \to H^k(X;\R)$, where we see $u$ as a real cohomology class and where $H^k_b(X;\R)$ denotes bounded cohomology, see \cite[Lemma 5 and Theorem 13]{KKM},
\item	$X$ is a closed oriented manifold, $k=\dim X$, $u=[X]^*$ is the dual of the fundamental class and $X$ is not dominated by any product of the form $P \times S^1$, where $P$ is a closed oriented manifold with $\dim P = \dim X-1$.
\end{itemize}
Concerning the last bullet, we recall that a closed manifold $N$ is \emph{dominated by $M$} if there exists a smooth map $f:M \to N$ of non-zero degree,  where $M$ is a closed manifold of the same dimension.
Results on the domination of manifolds by products have been obtained by D. Kotschick and C. L\"oh in \cite{KotschickLoeh} and by Kotschick and C. Neofytidis in \cite{KotschickNeofyt}, see also \cite[Section 4]{delaHarpe}.
\end{remark}

Theorem \ref{TheoremPullbackWeight} can now be used to derive lower bounds on topological complexity. The following result was obtained in this way by Neofytidis, improving previous results by the author from \cite[Section 6]{MescherSC}.

\begin{theorem}[{\cite[Theorem 1.1]{Neofyt}}]
\label{TheoremNeofyt}
Let $M$ be a closed oriented manifold of dimension $n \geq 3$. Assume that there exists a non-zero cohomology class $u \in H^k(M \times M;\Q)$ for some even $k$ with $1<k<n$, such that $f^*u=0$ for all continuous maps $f\colon P \times S^1 \to M$ for which $P$ is a closed oriented manifold with $\dim P=n-1$. Then
\begin{enumerate}[a)]
	\item $\tc(M)\geq 5$ if $n$ is odd,
	\item $\tc(M) \geq 6$ if $n$ is even,
	\item $\tc(M) \geq 8$ if $n=2k$ and $u\cup u \neq 0$.
\end{enumerate}
\end{theorem}

\subsection{Symplectic manifolds}
 Given a smooth manifold $M$, we recall that a \emph{symplectic form} is a closed and non-degenerate differential two-form  $\omega \in \Omega^2(M)$. Here,  \emph{closed} means that $d\omega = 0$ with respect to the exterior differential, while \emph{non-degenerate} means that $\omega_x\colon T_xM \times T_xM \to \R$ is a non-degenerate bilinear form for each $x \in M$.  A \emph{symplectic manifold} $(M,\omega)$ is a smooth manifold $M$ with a symplectic form $\omega \in \Omega^2(M)$. It is a basic result from symplectic geometry that every symplectic manifold is even-dimensional. In the case that $M$ is closed and $2n$-dimensional, it follows from the non-degeneracy of $\omega$ that the $n$-fold wedge product $\omega^n= \omega \wedge \dots \wedge \omega \in \Omega^{2n}(M)$ is a volume form on $M$. For the de Rham cohomology class of $\omega$, this particularly yields that its $n$-th power
$ [\omega]^n = [\omega^n] \in H^{2n}(M;\R)$ is non-zero. Thus, \emph{any} closed $2n$-dimensional symplectic manifold satisfies ${\sf {cl}}_{\R}(M)\geq n$. 

In the main result of \cite{RudyakOprea}, it is shown by Y. Rudyak and J. Oprea that if a closed symplectic manifold $(M,\omega)$ is \emph{symplectically aspherical}, i.e. if $\int_{S^2} f^*\omega =0$ for all continuous $f\colon S^2 \to M$, then its Lusternik-Schnirelmann category satisfies 
$$\cat (M) = \dim M.$$
In general, the Lusternik-Schnirelmann category of a manifold is only bounded from above by its dimension, so for symplectically aspherical manifolds, this upper bound is indeed attained.

However, the topological complexity of a symplectically aspherical manifold does not necessarily need to attain the dimensional bound given by \eqref{EqUpperTCdimconn}, as the $2$-torus is symplectically aspherical with respect to its standard volume form, but satisfies $\tc(T^2)=2<4$. In \cite{GrantMescher}, M. Grant and the author were able to define a slightly stricter condition which provides an analogue of \cite{RudyakOprea} for TC.

\begin{theorem}[{\cite[Theorem 1.2]{GrantMescher}}]
	Let $(M,\omega)$ be a closed symplectic manifold. If $\omega$ is \emph{symplectically atoroidal}, i.e. if $\int_{T^2}f^*\omega=0$ for each continuous $f\colon T^2 \to M$, then 
	$$\tc(M)=2\dim M.$$
\end{theorem}
\begin{proof}
	Put $u:= [\omega] \in H^2(M;\R)$ and $2n := \dim M$. Since $S^1$ is the only closed oriented one-dimensional manifold, it follows directly from Theorem \ref{TheoremPullbackWeight} that $\mathrm{wgt}_{\tc}(\bar{u})\geq 2$ if $(M,\omega)$ is symplectically atoroidal. Since, as mentioned above,  $u^n \neq 0 \in H^{2n}(M;\R)$, one computes that
	$$\bar{u}^{2n} = (1 \times \omega - \omega \times 1)^{2n} = (-1)^n \binom{2n}{n} u^n \times u^n \neq 0,$$
since all other terms that occur in the product contain terms of the form $u^k$ for some $k >n$, so they have to vanish for dimensional reasons. Thus, we derive from Theorem \ref{TheoremWeights}.c) that
$$\mathrm{wgt}_{\tc}(\bar{u}^{2n}) \geq 2n \cdot \mathrm{wgt}_{\tc}(\bar{u})=4n= 2 \dim M.$$
The claim follows immediately from Theorem \ref{TheoremWeights}.a).
\end{proof}

\begin{remark}
 Recently, R. Orita has shown in \cite{Orita} that the atoroidality condition on the symplectic form can be slightly weakened to still produce cohomology classes of TC-weight two in a completely analogous way. Orita studies so-called toroidally monotone symplectic manifolds $(M,\omega)$ of finite Kodaira dimension and shows that if additonally $\dim M =4$, then $\tc(M)=8$.
\end{remark}

	It seems a reasonable guess to the author that the values of Lusternik-Schnirelmann category and TC of symplectic manifolds are related to the existence of certain torus actions and we want to briefly discuss certain results undermining this guess. 
	
	 As seen in \cite[Proposition 1.3]{GrantMescher}, a symplectically atoroidal manifold does not admit any non-trivial \emph{symplectic} $S^1$-action, while by \cite[Theorem 4.16]{LuptonOpreaCsymp} a symplectically aspherical manifold does not admit any non-trivial \emph{Hamiltonian} $S^1$-action, but might admit more general symplectic ones. 
	 See \cite[Chapter 3]{AudinTorus} or \cite[Section 7.6]{FOT} for discussions of both notions.
	
	It was shown by J. Oprea and J. Walsh in \cite{OpreaWalsh} that if there is an effective Hamiltonian $T^k$-action on $M$, where $k \in \{1,2,\dots,\frac12\dim M\}$, such that each orbit of the action is contractible in $M$, then 
	$$\cat (M) \leq \dim M - k.$$
As far as the author is aware, there is no comparable upper bound for topological complexity in the presence of symplectic torus actions. However, it can be derived from \cite[Corollary 5.3]{GrantTCfibr} that whenever a manifold $M$ is equipped with a locally smooth and semi-free non-trivial $T^k$-action, then 
$$\tc(X) \leq 2 \dim M - k.$$

A reasonable question for future research thus seems to be the following.

\begin{question}
	Let $(M,\omega)$ be a closed symplectic manifold admitting an effective symplectic $T^k$-action, where $1 \leq k \leq \frac12 \dim M$.  Does it hold that $\tc(M) \leq 2 \dim M -k$?
\end{question}

 \subsection{Small values of TC for manifolds} Let $X$ be a topological space. An elementary argument shows that $\tc(X)=0$ holds if and only if $X$ is contractible. Moreover, it is a well-established result by M. Grant, G. Lupton and J. Oprea from \cite{GLO} that every topological space $X$ with $\tc(X)=1$ has the homotopy type of an odd-dimensional sphere. The obvious next question is what one can say about spaces of topological complexity two. We want to discuss this question under the additional assumption that $X$ is a closed manifold. 

The case of a closed surface $S$ has been fully understood and it is known that $\tc(S)=2$ if and only if $M=S^2$ or $M=T^2$. Therefore, we shall assume throughout the following that $M$ is a closed manifold with $\dim M \geq 3$.  If $\tc(M)= 2$, then necessarily $\cat(M) \leq 2$ as well, which implies by a theorem of A. Dranishnikov, M. Katz and Y. Rudyak on Lusternik-Schnirelmann categories of manifolds, see \cite[Theorem 1.1]{DKR}, that $\pi_1(M)$ is trivial or a free group. 

It turns out that the condition $\tc(M)\leq 2$ is much more restrictive as can be seen in  the cohomology rings of $M$. The following conditions on cohomology are a direct consequence of $\tc(M)\leq 2$  for arbitrary coefficient rings $R$:
\begin{itemize}
	\item All triple cup products of classes of positive degree must vanish: one checks that if there were classes $u_1,u_2,u_3\in \widetilde{H}^*(M;R)$ with $u_1\cup u_2 \cup u_3 \neq 0$, then the product $\bar{u}_1 \cup \bar{u}_2 \cup \bar{u}_3$ would be non-zero as well. 
Then $\mathrm{wgt}_{\tc}(\bar{u}_1 \cup \bar{u}_2 \cup \bar{u}_3) \geq 3$, which would yield $\tc(M) \geq 3$, hence a contradiction. 
	\item If there is a class of TC-weight two, then its cup product with every zero-divisor must vanish, as otherwise one would likewise obtain a class of TC-weight three.
\end{itemize}
The reader can easily see that this puts strong restrictions on the cup product structures of $M$. A more sophisticated analysis of the cohomological restrictions for closed \emph{oriented} manifolds whose TC is at most two was carried out by P. Pave\v{s}i\'{c} in \cite{PavesicManif}. In the following result, we use the usual convention in graded rings that $x_i$ shall always denote a generator of degree $i$. We further denote an exterior algebra by $\Lambda$ and let $\mathbb{F}_2$ be the field with two elements.
\begin{theorem}[{\cite{PavesicManif}}]
	Let $M$ be a closed oriented manifold with $\tc(M) \leq 2$. Then $\pi_1(M)$ is trivial or a free group and one of the following holds:
	\begin{enumerate}[(i)]
		\item $M$ has the homotopy type of $S^k$ for some $k \in \mathbb{N}$,
		\item  $H^*(M;\Z) \cong \Lambda (x_i,x_j)$ as rings for some odd numbers $i$ and $j$,
		\item there exists $k \in \{1,2,\dots,\dim M\}$, such that $H_i(M;\Z)=\{0\}$ if $i \notin \{0,k,\dim M\}$ and such that $H^*(M;\mathbb{F}_2) \cong \Lambda (x_k,x_{k+1}) \otimes \mathbb{F}_2$ as rings.
	\end{enumerate}
\end{theorem}

In Subsection 1.1 we have derived lower bounds for TC using certain pullbacks to products. Conversely, one can use the corresponding results to derive the domination of a manifold by a product if its topological complexity is sufficiently small.

\begin{theorem}[{\cite[Propositions 6.2 and 6.3, Theorem 6.6]{MescherSC}}] Let $M$ be a closed oriented manifold and assume that one of the following holds:
\begin{enumerate}[(i)]
\item $\tc(M)\leq 2$,
	\item  $M$ is even-dimensional and $\tc(M) \leq 3$,  
\item $\dim M = 3$ and $\tc(M) \leq 3$.
\end{enumerate}
Then $M$ is a rational homology sphere or $M$ is dominated by a manifold of the form $P \times S^1$, where $P$ is a closed oriented manifold with $\dim P = \dim M -1$.
\end{theorem}


\begin{remark}
CW complexes of topological complexity two have also been studied by A. Boudjaj and Y. Rami in \cite{BoudjajRami}. There the authors have worked out properties of minimal CW decompositions under certain additional conditions on the cohomology of the complex. 
\end{remark}

\subsection{Motion planning in work spaces} As mentioned in the introduction, manifolds are often used to model configuration spaces of robots or other mechanical systems. However, if one is only interested in the concrete task the robot is carrying out, one should study motion planning in the \emph{work space} of a robot, i.e. the space that formalizes all of the possible positions of its end effector (EE). The physical position of the EE is often represented by the position of a particular point on the end effector as a point in $\R^3$ plus a representation of its \emph{orientation}, seen as a positive orthonormal basis in $\R^3$ or, equivalently, a matrix in $SO(3)$, see \cite[Section 2]{Pfalzgraf}. Assume now that there are additional holonomic constraints on the position of the EE, which occur for example if it is moving on an oriented surface $S\subset\R^3$. Here we assume that the size of the robot is negligible. If $x \in S$ is the position and $(b_1,b_2,b_3)$ is the orthonormal basis associated with the EE, the condition that it is bound to the surface yields that  one of the basis vectors, say $b_3$, is normal to the surface. Then $b_1$ and $b_2$ are tangent to the surface, $b_1,b_2 \in T_xS$. If we equip $S$ with the metric it inherits from the Euclidean metric of $\R^3$, then $(b_1,b_2)$ indeed forms an orthonormal basis of $T_xS$ which fully determines the orientation of the EE, since the side of $S$ that the EE is on is considered as fixed. 

In \cite{TCframe}, the author has suggested the following abstraction of the situation: the position of the end effector should be represented by a point in a smooth manifold that is equipped with a Riemannian metric, thus allowing for measurements of distances and lengths, and a positive orthonormal basis of the tangent space at that point with respect to the inner product given by the Riemannian metric. This is reflected in a well-known construction from differential geometry, namely the \emph{special orthonormal frame bundle} or simply the \emph{frame bundle of $M$}, where $M$ is an oriented Riemannian manifold, given by 
$${\small F(M,g) = \{(x,b_1,\dots,b_n) \ | \ x \in M, \ (b_1,\dots,b_n)\text{ is a positive orthonormal basis of }T_xM\}.}$$	
The obvious projection $p:F(M,g) \to M$ is then a principal $SO(n)$-bundle over $M$. In terms of the frame bundle, we can formulate what we call the \emph{oriented motion planning problem}: \medskip 

\emph{Given $x,y \in M$, a positive orthonormal basis $B_1$ of $T_xM$ and a positive orthonormal basis $B_2$ of $T_yM$, find a continuous path $\gamma\colon [0,1] \to F(M,g)$ with $\gamma(0)=(x,B_1)$ and $\gamma(1)=(y,B_2)$.}\medskip 

As the reader will easily realize, this is nothing but the original topological motion planning problem suggested by Farber applied to a frame bundle, so to study the complexity of oriented motion planning, we have to study the value of $\tc(F(M,g))$.

To find an upper bound for this value, we note that given a principal $G$-bundle $p:P \to B$ for a smooth Lie group $G$, the fact that $P$ is equipped with a free $G$-action together with a result of M. Grant, see \cite[Corollary 5.3]{GrantTCfibr}, yields that the upper bound on $\tc(P)$ from \eqref{EqUpperTCdimconn} can be improved to 
$$\tc(P) \leq 2 \dim B + \dim G.$$
In the case of a frame bundle over an $n$-dimensional manifold we obtain
$$\tc(F(M,g)) \leq 2n + \dim SO(n) = \frac{n(n+3)}{2}.$$
 If $M$ is parallelizable, i.e. if its tangent bundle is a trivial vector bundle, then $F(M,g)$ will be trivial as well and the standard inequality for TC of product spaces yields
$$\tc(F(M,g))\leq \tc(M)+ \tc(SO(n)) = \tc(M)+\cat(SO(n)),$$
using that $\tc(G)=\cat(G)$ for any path-connected topological group $G$, see \cite[Lemma 8.2]{FarberTC2}. In fact, the exact value of $\cat(SO(n))$ is only known for $n \leq 10$, see \cite[p. 11]{TCframe} for details and references.

It turns out that under certain conditions, the TC of a frame bundle has a simple lower bound. 

\begin{theorem}[{\cite[Theorem 1]{TCframe}}]
\label{TheoremFrame}
Let $(M,g)$ be an oriented $n$-dimensional Riemannian manifold. If one of the following holds:
\begin{enumerate}[(i)]
	\item the inclusion of a fiber induces a surjection $H^*(F(M,g);K) \to H^*(SO(n);K)$ for some field $K$ whose characteristic is not two,
	\item $M$ admits a spin structure,
\end{enumerate}	
then $\tc(F(M,g)) \geq n-1$.
\end{theorem}

Condition (i) of Theorem \ref{TheoremFrame} is known under the term that the bundle $F(M,g) \to M$ is  \emph{totally non-cohomologous to zero}, see \cite{FOT}.  If this condition holds, then the Leray-Hirsch theorem will be  applicable to $F(M,g) \to M$. Using this theorem, one uses the lower bound for TC by zero-divisor cup length to derive the result. Condition (ii) implies that $F(M,g)$ is double-covered by a principal $\mathrm{Spin}(n)$-bundle. Using this bundle, one bounds $\tc(F(M,g))$ from below by the sectional category of a certain double cover of $F(M,g) \times F(M,g)$, which is in turn estimated from below by $\dim M$ using cohomological methods. 

The author suggests the following question, as he is not aware of any counterexample to the assertion.

\begin{question}
Does the special orthonormal frame bundle of every oriented Riemannian manifold satisfy $\tc(F(M,g))\geq \dim M-1$?
\end{question}

\subsection{Connected sums} Consider two $n$-dimensional manifolds $M_1$ and $M_2$, where $n \in \mathbb{N}$.  We first remove an open $n$-ball both from $M_1$ and from $M_2$  and then glue both manifolds along the boundaries of the removed balls. The resulting manifold, which is unique up to isotopy, is called the \emph{connected sum of $M_1$ and $M_2$} and usually denoted by $M_1\# M_2$, see \cite[Example 9.31]{LeeSmooth}. 
Iterating this construction, one can define the connected sum of $k$ manifolds of the same dimension $M_1,M_2,\dots,M_k$, as 
$$\#_{i=1}^k M_i = M_1 \#M_2 \# \dots \#M_k.$$
As an example, a closed oriented surface of genus $g \in \mathbb{N}$ is obtained by taking the connected sum $S_g = \#_{i=1}^g T^2$ of $g$ copies of $T^2$. In the context of motion planning, this construction may occur if the space in which a robot is working is extended or merged with another one.

A natural question is how the complexity of motion planning in a connected sum is related to the complexity of its summands. It seems a reasonable guess that motion planning does not become simpler after taking connected sums as one removes a contractible subset from a manifold and glues in another manifold instead.  In the case that both manifolds are simply connected, the hypothesis has been confirmed by A. Dranishnikov and R. Sadykov in \cite{DSconnectedsum}.

\begin{theorem}[{\cite[Theorem 2 and Proposition 18]{DSconnectedsum}}] 
\label{TheoremTCconnsum}
Let $M_1$ and $M_2$ be two closed oriented manifolds with $\dim M_1=\dim M_2$. If $M_1$ is simply connected, then 
$$\tc(M_1 \# M_2) \geq \tc(M_1).$$
In particular, if $M_2$ is simply connected as well, then 
 $$\tc(M_1 \# M_2) \geq \max\{\tc(M_1),\tc(M_2)\}.$$
\end{theorem}
The proof of this result makes use of several results from homotopy theory and shall not be reiterated here. Dranishnikov and Sadykov further apply Theorem \ref{TheoremTCconnsum} to show that $\tc(M)=4$ for any closed oriented simply connected manifold $M$ with $\dim M=4$ and $M \not\cong S^4$. The reader might now think of the following natural question which, to the author's best knowledge, has not been answered yet.

\begin{question}
	Does the assertion of Theorem \ref{TheoremTCconnsum} still hold without the hypothesis that the manifolds are simply connected?
\end{question}

\begin{remark}
C. Neofytidis has obtained explicit lower bounds for certain connected sums as consequences of the above Theorem \ref{TheoremNeofyt}.  We refer to \cite[Corollaries 1.2 and 1.3]{Neofyt} for the exact results.
\end{remark}
To conclude our overview of connected sums, we want to briefly mention a result by D. Cohen and L. Vandembroucq. We recall that the topological complexity of real projective spaces has been computed by M. Farber, S. Tabachnikov and S. Yuzvinsky in \cite{FTY}, see also \cite[Section 4.8]{FarberBook} for an outline of the proof and the explicit values of $\tc(\R P^n)$ for all $n \leq 23$. While these computations show that $\tc(\R P^n)$ lies strictly below its upper bound \eqref{EqUpperTCdimconn} for each $n \leq 23$, the connected sum of multiple copies of  $\R P^n$ behaves differently. 

\begin{theorem}[{\cite[Theorem 1.3]{CVconnected}}]
	Let $k,n \in \mathbb{N}$ with $k \geq 2$ and $n \geq 2$. Then 
	$$\tc\big(\#_{i=1}^k\R P^n\big) = 2n.$$
\end{theorem}

It seems to be a common phenomenon that the value of TC of a connected sum very often reaches its upper bound given by \eqref{EqUpperTCdimconn} even if the individual summands do not. A full explanation for this observation is still to be found.

\subsection{Further results} There is a plethora of results on the topological complexity of manifolds that is not considered in this survey and the choice of topics discussed here is obviously biased by the author's research interests. We feel obliged to at least mention some results that we did not explain in detail. The reader should keep in mind that the following list is by no means exhaustive.
\begin{itemize}
\item In \cite{HigherTCSeifert}, the authors compute lower bounds for the TC and for sequential TCs of large classes of three-dimensional Seifert manifolds. They show that for a large class of Seifert manifolds $M$, it holds that $\tc(M) \in \{5,6\}$ using weights of cohomology classes and computations of cohomology rings for Seifert manifolds.  
\item The authors of \cite{Goubault} introduce the notion of \emph{directed topological complextity} of d-spaces, which connects topological robotics with the field of directed algebraic topology. As discussed in \cite[Section 2]{Goubault}, any smooth manifold possesses a d-space structure.
Directed algebraic topology has various connections to the theory of concurrent processes and to control theory. 
\item Several considerations on connections between amenable category and topological complexity are discussed in \cite{CLM}. The amenable category $\cat_{\mathrm{Am}}(X)$ of a space $X$ is the minimal cardinality of an amenable cover of $X$. The authors show in \cite[Section 8]{CLM} that certain manifolds satisfy the estimate $\cat_{\mathrm{Am}}(X \times X) \leq \tc(X)$ and suggest to determine for which topological spaces this inequality holds.
\item In the introduction to \cite{BaryshPerplex}, Y. Baryshnikov suggests to study the relations between TC, which measures discontinuities of motion planning, and topological perplexity, which is introduced in loc. cit. and measures discontinuities of feedback stabilization in control systems.
\item Lower and upper bounds for the TC of real Grassmann manifolds are obtained by M. Radovanovi\'{c} in \cite{Radovanovic}. In particular, the author obtains the explicit values of certain zero-divisor cup lengths of real Grassmannians. 
\end{itemize}

\section{Length-minimizing motion planning in Riemannian manifolds}

The geodesic complexity of a metric space has been developed first by D. Recio-Mitter in \cite{RecioMitter} and is introduced and presented in the contribution of  D. Davis to this volume. We note that geodesic complexity is defined for an arbitrary geodesic space, i.e. a metric space in which any two points are connected by a length-minimizing path.

In this section, we want to study the geodesic complexity of \emph{complete Riemannian manifolds}, which has been investigated by M. Stegemeyer and the author in \cite{MSGC} and \cite{GCfiber}. By the Hopf-Rinow theorem, see e.g. \cite[Theorem 6.19]{LeeRiem} or \cite[Theorem 5.7.1]{Petersen}, complete Riemannian manifolds are examples of geodesic spaces. We want to discuss how methods from Riemannian geometry can be used to derive lower and upper bounds of their geodesic complexities.  A reader unfamiliar with some of the notions from Riemannian geometry used in this subsection can find detailed expositions in the textbooks  \cite{LeeRiem} and \cite{Petersen}.  

\subsection{Geodesics and cut loci in Riemannian manifolds} Let $(M,g)$ be a complete Riemannian manifold. A geodesic in $M$ is a curve $\gamma \in C^1(I,M)$, where $I$ is an open interval, which satisfies 
$$\nabla^g_{\gamma'(t)}(\gamma'(t))=0, $$
where $\nabla^g$ denotes the covariant derivative induced by $g$.

 The Hopf-Rinow theorem implies that every geodesic in $M$ might be extended as a geodesic to the whole real line, so we assume throughout the following that all geodesics are already given as curves $\R \to M$. 
 It is well-known that geodesics in $M$ are \emph{locally} length-minimizing curves, i.e. that their restrictions to sufficiently small intervals are length-minimizing, but they may lose this property when considered on arbitrarily large intervals. 
 
In the following, we call the restriction of a geodesic to $[0,1]$ a \emph{geodesic segment}. Given $p,q \in M$, we call a geodesic segment $\gamma:[0,1] \to M$ with $\gamma(0)=p$ and $\gamma(1)=q$  a \emph{minimal geodesic segment from $p$ to $q$} if 
$$L_g(\gamma) \leq L_g(\alpha) \qquad \forall \alpha \in C^{1}([0,1],M)\ \  \text{ with } \ \ \alpha(0)=p, \ \alpha(1)=q.$$ Here, $L_g:C^1([0,1],M) \to \R$, $L_g(\gamma) = \int_0^1 \|\gamma'(t)\|_{\gamma'(t)}\, dt$, denotes the length functional associated with the Riemannian metric $g$.

The key advantage of Riemannian manifolds over arbitrary geodesic spaces when it comes to geodesic motion planning is that \emph{the behaviour of geodesics in $M$ can be studied from the tangent bundle of $M$}, which is a smooth vector bundle. To make this more precise, we first introduce some additional notation from Riemannian geometry.

For each $p \in M$ we let $\exp_p:T_pM \to M$ denote the Riemannian exponential map of $g$ and we recall that for any $p \in M$ and $v \in T_pM$, the curve $\R \to M$, $t \mapsto \exp_p(tv)$, is a well-defined geodesic. 
We further consider the \emph{extended exponential map of $M$}, given as
$$\mathrm{Exp}\colon TM \to M \times M, \qquad \mathrm{Exp}(p,v) = (p, \exp_p(v)).$$
We next recall the definition of cut loci and tangent loci.

\begin{definition} Let $(M,g)$ be a Riemannian manifold.
\begin{enumerate}[a)]
	\item Let $p \in M$, let $\gamma:\R \to M$ be a geodesic with $\gamma(0)=p$ and put 
	$$t_*:= \sup\{ t \in (0,+\infty) \ | \ \gamma|_{[0,t]} \text{ is length-minimizing}\}.$$
	A point in $M$ of the form $\gamma(t_*)$ is then called a \emph{cut point of $p$}. The tangent vector $t_* \cdot \gamma'(0) \in T_pM$ is called a \emph{tangent cut point of $p$}.
	\item The \emph{cut locus of $p\in M$} is given by 
$$ \mathrm{Cut}_p(M)= \{q \in M \ | \ q \text{ is a cut point of } p\}$$
and the \emph{tangent cut locus of $p$} by 
$$\widetilde{\mathrm{Cut}}_p(M)= \{v \in T_pM \ | \ v \text{ is a tangent cut point of } p\}.$$
\item The \emph{total cut locus of $M$} and the \emph{total tangent cut locus of $M$} are given by 
 $$\mathrm{Cut}(M) = \bigcup_{p \in M} (\{p\} \times \mathrm{Cut}_p(M)) \subset M \times M, \qquad \widetilde{\mathrm{Cut}}(M) = \bigcup_{p \in M} \widetilde{\mathrm{Cut}}_p(M) \subset TM.$$ 	
 \end{enumerate}
\end{definition}
		
It holds by construction and by definition of Riemannian exponential maps that
$$\exp_p\big(\widetilde{\mathrm{Cut}}_p(M)\big) = \mathrm{Cut}_p(M) \qquad \forall p \in M,$$
from which one derives that
\begin{equation}
\label{EqExpCut}
\mathrm{Exp}\big(\widetilde{\mathrm{Cut}}(M)\big)=\mathrm{Cut}(M).
\end{equation}

\begin{remark}
\label{RemarkCutLoci}
\begin{enumerate}
	\item We note that the definitions of cut loci in metric and Riemannian geometry are slightly different from another. To specify, we briefly recall the Riemannian definition. Let $M$ be a Riemannian manifold and let $p \in M$. A point $q$ lies in $\mathrm{Cut}_p(M)$ if and only if one of the following holds:
	\begin{itemize}
		\item $q$ is a conjugate point of $p$, i.e. there is a minimal geodesic segment $\gamma$ from $p$ to $q$ along which there exists a non-zero Jacobi field which vanishes in $p$ and $q$, 
		\item there exists more than one minimal geodesic segment from $p$ to $q$. 
	\end{itemize}
In metric geometry the cut locus of a point is usually defined as the set of points with the second of these properties.
\item The cut locus of a point in a complete Riemannian manifold is always closed and of Lebesgue measure zero, see \cite[Theorem 10.34.(a)]{LeeRiem}. 
\item Unfortunately, these are the only general results about the structure of cut loci as they might be wildly different subsets of $M$ for different choices of Riemannian metrics, or even of two different points with respect to the same metric. A result which illustrates the possible complications was given by H. Gluck and D. Singer in \cite[Theorem A]{GluckSinger}. They have shown that for \emph{any} manifold $M$ with $\dim M \geq 2$ and for any $p \in M$ there exists a Riemannian metric on $M$ for which $\mathrm{Cut}_p(M)$ is not triangulable. 
\end{enumerate}
\end{remark}

\subsection{Geodesic complexity of Riemannian manifolds} Without further ado, we want to recall the definition of geodesic complexity as introduced by D. Recio-Mitter in \cite{RecioMitter}, but formulated in the Riemannian setting.
\begin{definition}[{\cite[Definition 1.7]{RecioMitter}}]
Let $(M,g)$ be a complete Riemannian manifold and let $GM \subset PM$ denote the subspace of minimal geodesic segments.
\begin{enumerate}[a)]
	\item Let $A \subset M \times M$. We call a map $s:A \to GM$ a \emph{geodesic motion planner on $A$} if 	$$(s(p,q))(0)=p \qquad \text{and} \qquad (s(p,q))(1)=q \qquad \forall (p,q) \in A.$$
	\item  The \emph{geodesic complexity of $(M,g)$} is denoted by ${\sf {GC}}(M,g)$ and defined as the minimal value of $r \in \mathbb{N}$, such that there is a decomposition $M \times M = A_0 \sqcup A_1 \sqcup \dots \sqcup A_r$ into locally compact subsets with the property that for each $i \in \{0,1,\dots,r\}$ there exists a continuous geodesic motion planner on $A_i$.
	\end{enumerate}
\end{definition}


We want to restrict our attention to \emph{closed} Riemannian manifolds, where we note that every closed Riemannian manifold is complete, see \cite[Corollary 6.22]{LeeRiem}. If $M$ is a closed manifold, then each ray in $T_pM$ emanating from the origin contains a unique element of $\widetilde{\mathrm{Cut}}_p(M)$. Moreover, these points assemble in such a way that $\widetilde{\mathrm{Cut}}_p(M)$ is homeomorphic to a $(\dim M-1)$-dimensional sphere and if we put
$$K_p =\{ t \cdot v \in T_pM \ | \ v \in \widetilde{\mathrm{Cut}}_p(M), \ t \in [0,1]\} \qquad \forall p \in M, $$
and set $K:= \bigcup_{p \in M} K_p$, then $K$ is a bundle over $M$ whose fiber is homeomorphic to a ($\dim M$)-dimensional ball. In terms of geodesic motion planning, the crucial observation underlying the results of this section lies in the following statement. 
\begin{lemma}
\label{LemmaGeod}
Let $M$ be a closed Riemannian manifold and let $A \subset M \times M$. The following statements are equivalent:
\begin{enumerate}[(i)]
	\item There exists a continuous geodesic motion planner $s: A \to GM$.
	\item There exists a continuous local section $\sigma: A \to K$ of $\mathrm{Exp}|_K: K \to M \times M$.
\end{enumerate}
\end{lemma}
\begin{proof}
	(ii) $\Rightarrow$ (i): \enskip By definition of $\mathrm{Exp}$, it follows that $s(p,q) \in K_p$ for all $(p,q) \in A$. Thus, the following map is well-defined and by the properties of exponential maps its image indeed lies in $GM$:
	$$s\colon A \to GM, \qquad (s(p,q))(t)= \exp_p(t \cdot \sigma(p,q)) \quad \forall t \in [0,1], \ (p,q) \in A,$$
	Clearly, $s$ is continuous and we compute for all $(p,q) \in A$ that 
	$$(\pi \circ s)(p,q)= ((s(p,q))(0),(s(p,q))(1))= (\exp_p(0),\exp_p(\sigma(p,q)))= (p,q).$$
	Thus, $s$ is a continuous geodesic motion planner. 
	
	(i) $\Rightarrow$ (ii): \enskip It follows from a standard computation carried out in \cite[Proposition 3.1]{MSGC} that the map $\nu:GM \to TM$, $\nu(\gamma)=\gamma'(0)$, is continuous, and from the properties of minimal geodesics that $\nu(GM) \subset K$.  If we define $\sigma:A \to K$	, $\sigma := \nu \circ s$, then $\sigma$ is again continuous and we obtain that
\begin{align*}
(\mathrm{Exp}\circ \sigma)(p,q)&= \mathrm{Exp}(v(\sigma(p,q))= \mathrm{Exp}((s(p,q))'(0))\\
&= (p,\exp_p((s(p,q))'(0)))= (p,(s(p,q))(1))=(p,q),
\end{align*}
	where we have used the defining property of Riemannian exponential maps.
\end{proof}

This lemma immediately yields an alternative description of geodesic complexity.

\begin{cor}
\label{CorGCTM}
Let $(M,g)$ be a closed Riemannian manifold. Then 	${\sf {GC}}(M,g)$ is the minimal value of $n \in \mathbb{N}$, for which there exist $A_0,A_1,\dots,A_n \subset M \times M$ with 
$$A_0 \sqcup A_1 \sqcup \dots \sqcup A_n = M \times M,$$
such that for each $i \in \{0,1,\dots,n\}$ there exists a continuous local section $A_i \to K$ of $\mathrm{Exp}|_K$.
\end{cor}

Away from the total tangent cut locus  the restriction 
$$ \mathrm{Exp}|_{K \smallsetminus \widetilde{\mathrm{Cut}}(M)}: K \smallsetminus \widetilde{\mathrm{Cut}}(M) \to (M \times M)\smallsetminus \mathrm{Cut}(M),$$
is a well-defined homeomorphism, so its inverse is a continuous local section of $\mathrm{Exp}|_K$. Thus, by Lemma \ref{LemmaGeod}, there exists a continuous geodesic motion  planner on the complement of $\mathrm{Cut}(M)$. This motion planner was first studied and discussed by Z. B\l{}aszczyk and J. Carrasquel in \cite{BlaszCarras}. The difficulty in determining geodesic complexity lies in estimating the necessary number of local sections of $\mathrm{Exp}$ over $\mathrm{Cut}(M)$. 

\subsection{Fibered decompositions of cut loci} To derive upper bounds for the geodesic complexity of certain manifolds, we want to introduce a concept studied by M. Stegemeyer and the author in \cite{GCfiber}. \medskip 
By \eqref{EqExpCut}, we may consider the restriction 
$$F:= \mathrm{Exp}|_{\widetilde{\mathrm{Cut}}(M)}: \widetilde{\mathrm{Cut}}(M) \to \mathrm{Cut}(M).$$
Unfortunately, this restriction $F$ is not particularly well-behaved from a topological point of view. Since cut loci and tangent cut loci might be wildly different for two points in the same manifold, $F$ will in general not be a fibration. However, it turns out that we may occasionally decompose $F$ into a finite number of fibrations in the following sense. 
\begin{definition}[{\cite[Definition 3.1]{GCfiber}}]
A locally compact decomposition $\mathrm{Cut}(M) = A_1 \sqcup A_2 \sqcup \dots \sqcup A_r$ is called a \emph{fibered decomposition of $\mathrm{Cut}(M)$}, if the map $F_i:\widetilde{A}_i \to A_i$, $F_i := \mathrm{Exp}|_{\widetilde{A}_i}$, is a fibration for each $i \in \{1,2,\dots,r\}$, where we put 
$$\widetilde{A}_i := \widetilde{\mathrm{Cut}}(M) \cap \mathrm{Exp}^{-1}(A_i).$$
\end{definition}
Using Corollary \ref{CorGCTM}, we can derive the following statement, which shows the usefulness of fibered decompositions in our context. 

\begin{theorem}[{\cite[Theorem 3.2]{GCfiber}}]
\label{TheoremGCUpper}
Let $(M,g)$ be a complete Riemannian manifold and let $\{A_1,\dots,A_r\}$ be a fibered decomposition of $\mathrm{Cut}(M)$. Then 
$${\sf {GC}}(M,g) \leq \sum_{i=1}^r\secat(\mathrm{Exp}|_{\widetilde{A}_i}\colon \widetilde{A}_i \to A_i)+r. $$
\end{theorem}
This result gives us the possibility of bounding ${\sf {GC}}(M,g)$ by sectional categories, which can in turn be estimated using a big topological toolbox. Upon first sight, the existence of a fibered decomposition of the total cut locus might seem like wishful thinking, but indeed there exist such fibered decompositions for certain classes of \emph{homogeneous} Riemannian manifolds, i.e. Riemannian manifolds whose isometry group acts transitively on them. 

Since isometries map geodesics to geodesics, one checks that if $\phi:M \to M$ is an isometry, then
$$\mathrm{Cut}_{\phi(p)}(M) = \phi(\mathrm{Cut}_p(M)) \qquad \forall p \in M. $$
Thus, if the isometry group of $(M,g)$ acts transitively on $M$, then all cut loci of points in $M$ are homeomorphic to one another. Assuming a little additional condition, we can use a decomposition of the cut locus of a particular point to construct a fibered decomposition of the total cut locus using isometries. 

\begin{definition}
\label{DefIsotropInv}
Let $(M,g)$ be a closed Riemannian manifold with isometry group $G$. Let $p \in M$ and let $G_p$ denote the isotropy group of the usual $G$-action on $M$ in $p$. We call a locally compact decomposition $\{B_1,\dots,B_r\}$ of $\mathrm{Cut}_p(M)$ \emph{isotropy-invariant} if $g \cdot B_i \subset B_i$ for all $g \in G_p$, $i \in \{1,2,\dots,r\}$. 
\end{definition}

\begin{lemma}
\label{LemmaIsotropInv}
	Let $(M,g)$ be a homogeneous closed Riemannian manifold with isometry group $G$. Let $p \in M$, such that $\mathrm{Cut}_p(M)$ admits an isotropy-invariant decomposition $\{B_1,\dots,B_r\}$, for which 
	$$\exp_p|_{\widetilde{B}_i}: \widetilde{B}_i\to B_i$$
	is a fibration for each $i \in \{1,2,\dots,r\}$, where $\widetilde{B}_i:= \exp_p^{-1}(B_i) \cap \widetilde{\mathrm{Cut}}_p(M)$. Then $\{A_1,\dots,A_r\}$ is a fibered decomposition of $\mathrm{Cut}(M)$, where 
$$A_i =  \{(g \cdot p, g \cdot b) \in M \times M \ | \ g \in G, \ b \in  B_i\} \qquad \forall i \in \{1,2,\dots,r\}.$$
\end{lemma}

This lemma gives us a machinery to produce fibered decompositions of total cut loci, which is applicable to several kinds of homogeneous manifolds.
\begin{example}
\label{ExampleIsoInv}
	\begin{enumerate}
		\item In \cite{Sakai1} and \cite{Sakai2}, T. Sakai has described a neat decomposition of the cut loci of an irreducible compact simply connected symmetric space in terms of a root system of the space. This decomposition turns out to fulfill the conditions of Lemma \ref{LemmaIsotropInv}, so we derive that the total cut locus of an irreducible compact simply connected symmetric space admits a fibered decomposition.
		\item Consider the three-dimensional lens space $L(p,1)=S^3/\Z_p$, where $p \geq 3$, with a metric of constant curvature as a homogeneous Riemannian manifold via the action of the Lie group $S^3\cong SU(2)$. It is shown in \cite[Section 5]{GCfiber} that the cut locus of a point in $L(p,1)$ admits an isotropy-invariant decomposition to which Lemma \ref{LemmaIsotropInv} is applicable.
		\item We consider the two-dimensional torus $T^2$ with a flat Riemannian metric inherited from $\R^2$ by viewing $T^2=\R^2/\Gamma$, where $\Gamma \subset \R^2$ is a lattice whose generating vectors are not orthogonal to each other. The cut locus and the tangent cut locus of a point $p \in T^2$ are described in \cite[Example 2.114.e)]{GHL}. Roughly, $\mathrm{Cut}_p(T^2)$ consists of two points $q_1$ and $q_2$ which are connected by three pairwise disjoint arcs. The tangent cut locus takes the form of a hexagon, in which any opposite edges are mapped onto the same connecting arc under the exponential map. Consider the decomposition $\mathrm{Cut}_p(T^2)=B_1 \sqcup B_2$, where $B_1=\{q_1,q_2\}$ and $B_2=\mathrm{Cut}_p(T^2)\smallsetminus \{q_1,q_2\}$. The preimages of $q_1$ and $q_2$ consist of three points each, the corners of the hexagon. $B_2$ consists of three pairwise disjoint arcs whose preimages each consist of two disjoint arcs. One checks without difficulties that this decomposition is indeed isotropy-invariant and satisfies the assumptions of Lemma \ref{LemmaIsotropInv}.
	\end{enumerate}
\end{example}

Applying Theorem \ref{TheoremGCUpper} to the fibered decompositions discussed in Example \ref{ExampleIsoInv}, one obtains the following upper bounds. Details of the computations are found in \cite{GCfiber}. 

\begin{itemize}
	\item (\cite[Theorem 4.6]{GCfiber}) \enskip Consider the complex projective space $\mathbb{C}P^n$ with the Fubini-Study metric. Then
	$${\sf {GC}}(\C P^n,g_{FS})=2n = \tc(\C P^n).$$
	\item (\cite[Example 4.5]{GCfiber}) \enskip Consider the complex Grassmann manifold $G_2(\C^4)$ as a compact symmetric space. Then 
	$${\sf {GC}}(G_2(\C^4),g_{\mathrm{sym}}) \leq 15.$$
	\item (\cite[Theorem 5.8]{GCfiber})\enskip  For the lens space $L(p,1)$ with a constant curvature metric $g_c$, we obtain $${\sf {GC}}(L(p,1),g_c)\leq 6.$$
	\item (\cite[Theorem 7.2]{MSGC})\enskip Given any flat metric $g_f$ on $T^2$, it holds that ${\sf {GC}}(T^2,g_f)=2$.
\end{itemize}

\subsection{Lower bounds for geodesic complexity} As we have seen in the first section, lower bounds for TC are often derived from cohomology rings of the spaces under investigation. The general philosophy is vaguely described by saying that certain cohomology classes provide obstructions to the existence of continuous motion planners. In the study of \emph{geodesic} complexity, one needs to find \emph{geometric obstructions} to the existence of continuous geodesic motion planners. An example of such an obstruction is given in the following statement, which was initially observed in \cite[Remark 3.17]{RecioMitter}.

\begin{proposition}[{\cite[Proposition 3.7]{MSGC}}]
\label{PropNonexist}
Let $(M,g)$ be a complete Riemannian manifold, let $p \in M$ and $q \in \mathrm{Cut}_p(M)$, such that there are two distinct minimal geodesic segments $\gamma_1,\gamma_2 \in GM$ from $p$ to $q$. Let $U$ be an open neighborhood of $q$. Then there is no continuous geodesic motion planner on $\{p\} \times U$.
	\end{proposition}
\begin{proof}
Let $v_1,v_2 \in T_pM$, $v_1 \neq v_2$, be given by $\gamma_i(t) = \exp_p(t\cdot v_i)$ for all $t \in [0,1]$, $i \in \{1,2\}$. Choose a sequence $(r_n)_{n \in \mathbb{N}}$ in $[0,1)$ with $\lim_{n \to \infty} r_n=1$ and such that $\gamma_{i}(r_n) \in U$ for all $n \in \mathbb{N}$, $i \in \{1,2\}$. Assume there was such a continuous geodesic motion planner $s\colon \{p\} \times U\to GM$.  By definition of the cut locus, $\gamma_i(r_n)\notin \mathrm{Cut}_p(M)$, such that $s(p,\gamma_i(r_n))$ is the unique minimal geodesic segment from $p$ to $\gamma_i(r_n)$ for all $i$ and $n$. One checks from the properties of the exponential map that this segment is given by
\begin{equation}
\label{EqMPGeod}
(s(p,\gamma_i(r_n)))(t)= \exp_p(t \cdot r_n \cdot v_i)) \qquad \forall t \in [0,1],\ n \in \mathbb{N}, \ i \in \{1,2\}. 
\end{equation}
The map $\nu\colon GM \to TM$, $\nu(\gamma)=\gamma'(0)$, that already occurred in the proof of Lemma \ref{LemmaGeod}, is continuous, so $\nu\circ s\colon \{p\}\times U \to TM$ must be a continuous local section of $\mathrm{Exp}$. However, we derive from \eqref{EqMPGeod} that
$$(\nu \circ s)(p, \gamma_i(r_n)) = r_n \cdot v_i \qquad \forall n \in \mathbb{N},\ i \in \{1,2\}$$
But then 
$$\lim_{n \to \infty} (\nu \circ s)(p, \gamma_1(r_n))= v_1 \neq v_2 = \lim_{n \to \infty} (\nu \circ s)(p, \gamma_2(r_n)),$$
contradicting the continuity of $\nu \circ s$. Thus, there is no such continuous geodesic motion planner.
\end{proof}

In the previous proposition, we have derived a non-existence result for geodesic motion planners from a property of the tangent cut locus $\widetilde{\mathrm{Cut}}_p(M)$, namely that there are two different tangent vectors that are mapped onto the same point by the exponential map.


This evokes the question how more sophisticated observations on the behaviour of the tangent cut loci can yield lower bounds on geodesic complexity.  In \cite{MSGC}, M. Stegemeyer and the author have introduced the notion of an \emph{inconsistent stratification} of the cut locus, motivated by the idea of finding geometric obstructions to the existence of continuous geodesic motion planners in tangent cut loci. 

Similar geometric obstructions that do not make use of tangent spaces, but are technically more intricate, were introduced by D. Recio-Mitter in \cite[Section 3]{RecioMitter} for more general metric spaces. Since the definition of inconsistent stratifications, see \cite[Definition 4.6]{MSGC} is a very technical one, we shall only present a special case which generalizes the condition used in Proposition \ref{PropNonexist} to some extent. \medskip 

In the following we outline a notion investigated by J.-I. Itoh and T. Sakai in \cite{ItohSakai}. We assume that $(M,g)$ be a closed $n$-dimensional Riemannian manifold and that $p \in M$ is given in such a way that $\mathrm{Cut}_p(M)$ does not contain any conjugate point of $p$. For each $k \in \mathbb{N}$ let $C_k \subset \mathrm{Cut}_p(M)$ consist of those points, for which there are precisely $k+1$ distinct minimal geodesic segments from $p$ to $q$, so that $\mathrm{Cut}_p(M)= \bigcup_{k \in \mathbb{N}} C_k$.

\begin{definition}
 We call the decomposition $\mathrm{Cut}_p(M)=\bigcup_{k \in \mathbb{N}} C_k$ \emph{non-degenerate} if the following holds for all $k \in \mathbb{N}$ and $q \in C_k$: \medskip

If $\gamma_0,\gamma_1,\dots,\gamma_k \in GM$ are the $k+1$ distinct minimal geodesic segments from $p$ to $q$, then $\{\gamma'_0(0),\gamma_1'(0),\dots,\gamma_k'(0)\}\subset T_pM$ is in general position, i.e. the vectors 
$$\gamma_1'(0)-\gamma'_0(0),\gamma_2'(0)-\gamma'_0(0),\dots,\gamma_k'(0)-\gamma'_0(0)$$ 
are linearly independent in $T_pM$. 
\end{definition}
Apparently, if the decomposition is non-degenerate, then for dimensional reasons $C_k=\emptyset$ for all $k >n$. Using a result by Itoh and Sakai, see \cite[Proposition 2.4]{ItohSakai}, one shows that the decomposition $(C_1,\dots,C_n)$ of a cut locus indeed forms an inconsistent stratification and obtains the following result.

\begin{theorem}[{\cite[Corollary 4.12]{MSGC}}]
\label{TheoremGCLower}
Let $(M,g)$ be a closed Riemannian manifold and let $p \in M$, such that $\mathrm{Cut}_p(M)$ does not contain any conjugate point of $p$. If the decomposition $(C_1,\dots,C_n)$ from above is non-degenerate with $C_n \neq \emptyset$, then 
	$${\sf {GC}}(M,g) \geq n.$$
\end{theorem}

\begin{remark}
Note that the lower bounds in this subsection are obtained from the study of the (tangent) cut locus \emph{of one particular point in $M$}. So far there is no generalization of the inconsistency or non-degeneracy condition to \emph{total} cut loci that might potentially provide stronger lower bounds for geodesic complexity. 	
\end{remark}

\section{Critical point theory and sectional category}

For almost a century, mathematicians have been studying connections between numerical homotopy invariants and critical points of functions on manifolds, starting with the work of L. Lusternik and L. Schnirelmann who introduced their eponymous category for, see \cite{LSbook}, and by M. Morse, see \cite{MorseCalcVarLarge}, whose theory has been studied and generalized ever since. 

Concerning topological robotics, a connection to critical point theory is given by so-called \emph{navigation functions}. The underlying idea, which first occurred in works of D. Koditschek and E. Rimon, is to carry out motion planning along paths whose traces are the traces of gradient flow trajectories of a differentiable function. In this section, we want to give a brief exposition of navigation functions and how they yield upper bounds on the complexity of motion planning in a manifold. We also want to take a slight detour from motion planning to discuss the notion of spherical complexities. The latter has been introduced by the author to study critical points of functions on loop and sphere spaces and applied to the closed geodesics
 problem in \cite{MescherExistGeod} and \cite{MescherSC}. We will see that spherical complexities indeed generalize TC in a certain sense and discuss the connection to navigation functions.
\medskip

Throughout this section, we will employ the following notation:
\begin{itemize} 
\item Given a topological space $X$, we will write $\pi_0(X)$ for the set of path-connected components of $X$.
\item Given a function $f\colon X \to \R$ and some $a \in \R$, we denote the closed sublevel set and the level set associated with $a$ by 
$$ f^a := \{x\in X \ | \ f(x)\leq a\}, \qquad f_a := \{x \in X \ | \ f(x)=a\}.$$
\item If $X$ is a smooth manifold and $f$ is differentiable, then we shall denote the set of critical points of $f$ by 
$$\mathrm{Crit}\, f = \{x \in X \ | \ df(x)=0\}.$$
\item We will further write $\cat_X(A)$ for the subspace category of $A \subset X$, see \cite[Definition 1.1]{CLOT}, and $\tc_X(B)$ for the subspace TC of $B \subset X \times X$. Here, with $i_B\colon B \hookrightarrow X\times X$ denoting the inclusion,  we let the latter be given as
$$\tc_X(B)= \secat(i_B^*\pi\colon i_B^*PX \to B),$$
i.e. the sectional category of the pullback of the path fibration via the inclusion, see also \cite[Section 4.3]{FarberBook}.
\end{itemize}
We further note that given a path-connected space $X$ and $A_1,\dots,A_r \subset X \times X$, an elementary computation shows that
\begin{equation}
\label{EqTCSubspace}
\tc_X\Big(\bigcup_{i=1}^r A_i \Big) \leq \sum_{i=1}^r \tc_X(A_i)+r-1	
\end{equation}

\subsection{Navigation functions and their generalizations}
A connection between critical point theory and topological complexity is given by the study of \emph{navigation functions}. This notion picks up an idea by D. Koditschek and E. Rimon who used critical points of Morse function on manifolds with boundary in robot navigation in \cite{KoditschekRimon} and \cite{RimonKoditStar}. They study the  motion planning problem for robots with respect to a fixed target point by considering a Morse function having a global minimum in the target point. 
The corresponding approach to topological complexity was suggested by M. Farber in \cite[Section 4.4]{FarberBook} and further elaborated upon and applied to lens spaces by A. Costa in \cite{Costa}. The main idea is to use a similar approach for flexible target points by studying Morse-Bott functions on $M \times M$ with special properties, assuming that the configuration space $M$ of the robot is a smooth manifold. The precise definition is as follows. 

\begin{definition}
	Let $M$ be a smooth manifold. A $C^2$-function $F:M \times M \to \R$ is called a \emph{navigation function} if it is a non-negative  Morse-Bott function with $F_0=F^{-1}(\{0\})=\Delta_M$.
\end{definition}

See \cite[Section 2.6]{NicolaescuMorse} for an overview of Morse-Bott theory. Since $\Delta_M$ consists of exactly those points in which $F$ attains its absolute minimum, $\Delta_M$ is a critical submanifold of every navigation function. 

\begin{lemma}
\label{LemmaStableDelta}
	Let $M$ be a smooth manifold, let $F:M \times M \to \R$ be a navigation function and let $g$ be a Riemannian metric on $M$. Let $W^s(\Delta_M)$ be the stable manifold of $\Delta_M$ with respect to the negative gradient of $F$ with respect to $g$. Then $W^s(\Delta_M)$ admits a continuous motion planner $s: W^s(\Delta_M)\to PM$.
\end{lemma}
\begin{proof}
	Let $\phi: \R \times M^2 \to M^2$ be the negative gradient flow of $F$ with respect to $g$ and define 
$$\gamma^1_{x,y},\gamma^2_{x,y}:[0,+\infty) \to M, \qquad (\gamma^1_{x,y}(t),\gamma^2_{x,y}(t)):= \phi(t,x,y)$$ 
for all $t \in [0,+\infty)$,  $x,y\in M$.
Then $\lim_{t \to + \infty} \phi(t,x,y) \in \Delta_M$ for any $(x,y) \in W^s(\Delta_M)$, or equivalently
	$$\lim_{t \to + \infty} \gamma_{x,y}^1(t) = \lim_{t\to +\infty} \gamma_{x,y}^2(t). $$
We want to reparametrize these trajectories to obtain paths suitable for motion planning purposes. Fix $(x,y) \in W^s(\Delta_M)$. For $i \in \{1,2\}$ we put 
$$\gamma_i:= \gamma_{x,y}^i, \qquad p := \lim_{t\to +\infty} \gamma_i(t) \qquad \text{and} \qquad E(\gamma_i) := \int_0^{+\infty} \|\gamma_i'(s)\|_g^2\, ds.$$ 
Here, $\|\cdot \|_g$ denotes the norm induced by $g$ on the respective tangent space of $M$. A standard computation shows that the integral $E(\gamma_i)$ converges, so the following map is well-defined:
	$$\tau_i: [0,+\infty) \to [0,1), \qquad \tau_i(t) = \frac{\int_0^t \|\gamma_i'(s)\|_g^2\, ds}{E(\gamma_i)}.$$
Apparently, $\tau_i$ is continuous, surjective and strictly increasing, hence a homeomorphism, and one checks that the following  is indeed a continuous path for each $i \in \{1,2\}$:
$$ \alpha^i_{x,y} \colon[0,1] \to M, \qquad \alpha^i_{x,y}(t) = \begin{cases}
	\gamma_i(\tau_i^{-1}(t)) & \text{if } t \in [0,1), \\
	p & \text{if } t=1.
\end{cases}$$
These paths satisfy $\alpha^1_{x,y}(0)=x$, $\alpha^2_{x,y}(0)=y$ and $\alpha^1_{x,y}(1)=\alpha^2_{x,y}(1)$. One checks from the continuity of the flow  that the thus-obtained motion planner
	$$s:W^s(\Delta_M) \to PM, \qquad s(x,y) = \alpha_{x,y}^1 * \bar{\alpha}^2_{x,y}$$
	is indeed continuous. Here, we let $*$ denote the concatenation of paths and let $\bar{\alpha} \in PM$ be given by $\bar{\alpha}(t)=\alpha(1-t)$ for all $\alpha \in PM$ and $t \in [0,1]$.
\end{proof}

Using the previous lemma, we can establish an upper bound on $\tc(M)$ by subspace TCs of critical manifolds of navigation functions. 
\begin{theorem}[{\cite[Theorem 4.32]{FarberBook}}]
\label{TheoremTCNavigation}
Let $M$ be a closed smooth manifold and let $F:M \times M \to \R$ be a navigation function whose critical values are given by $0<c_1<c_2 < \dots <c_r$. For each $i \in \{1,2,\dots,r\}$ we put $\mathcal{C}_i( F) := \pi_0(\mathrm{Crit}\, F \cap F_{c_i})$. Then 
$$\tc(M) \leq \sum_{i=1}^r \tc_M(\mathrm{Crit}\, F \cap F_{c_i})+r = \sum_{i=1}^r \max \{\tc_M(C) \ | \ C \in \mathcal{C}_i(F)\}+r.$$
\end{theorem}
\begin{proof}
By standard results from Morse-Bott theory, since $M$ is closed, hence complete,  $M \times M$  decomposes as
	$$ M \times M = \bigcup_{C \in \pi_0(\mathrm{Crit}\, F)} W^s(C) = W^s(\Delta_M) \cup \bigcup_{i=1}^r \bigcup_{C \in \mathcal{C}_i(F)} W^s(C).$$
	We derive from \eqref{EqTCSubspace} that 
	\begin{align*}
	\tc(M) &\leq 
	\sum_{i=1}^r \tc_M \Big(\bigcup_{C \in \mathcal{C}_i(F)} W^s(C) \Big)+r,
	\end{align*}
	where we used that $\tc_M(W^s(\Delta_M)=0$ by Lemma \ref{LemmaStableDelta}.
	By results from Morse-Bott theory, see e.g. \cite[Lemma 3.3]{AustinBraam}, it holds for all $i \in \{1,2,\dots,r\}$ and each $C \in \mathcal{C}_i(F)$ that $$\overline{W^s(C)} \subset W^s(C) \cup \bigcup_{j>i} \bigcup_{D \in \mathcal{C}_j(F)} W^s(D),$$ which particularly implies that
	$$\overline{W^s(C)}\cap W^s(C') = \emptyset \qquad \forall C,C' \in \mathcal{C}_i(F), \ i \in \{1,2,\dots,r\}.$$
It is a well-known trick in the study of topological complexity that if the closure of a domain of one continuous motion planner does not intersect the domain of another and vice versa, then the domains may be merged to form a domain of \emph{one} continuous motion planner. (Note that the definition of a continuous motion planner does not require its domain to be connected.)	Here, this showns that domains of continuous motion planners on different components of $\mathcal{C}_i(f)$ can be merged in the way we just described. Reducing the number of domains of continuous motion planners shows that 
	$$\tc_M \Big(\bigcup_{C \in \mathcal{C}_i(F)} W^s(C) \Big) = \max \{\tc_M(W^s(C)) \ | \ C \in \mathcal{C}_i(F)\}.$$
for each $i \in \{1,2,\dots,r\}$. By the stable manifold theorem for Morse-Bott functions $W^s(C)$ is homeomorphic to the normal bundle of $C$ in $M$. Moreover, one obtains a continuous retraction $q_C:W^s(C) \to C$ by sending each point to its limit point under the negative gradient flow, which by \cite[Corollary 4.26]{FarberBook} yields that $\tc_M(W^s(C))=\tc_M(C)$ for any $C \in \pi_0(\mathrm{Crit}\, F)$. Combining the previous observations shows the claim.
\end{proof}

Recently, the use of navigation functions has been picked up in the context of \emph{homotopic distance}. This notion was introduced by E. Mac\'{i}as-Virg\'{o}s and D. Mosquera Lois in \cite{MVMLhomdist}. Given two topological spaces $X$ and $Y$ and two continuous maps $f,g\colon X \to Y$, the \emph{homotopic distance of $f$ and $g$}, denoted by $D(f,g)$, is the minimal value of $n \in \mathbb{N}$, for which there exists an open cover $U_0,U_1,\dots,U_n$ of $X$, such that $f|_{U_j}\colon U_j \to Y$ is homotopic to $g|_{U_j}$ for each $j \in \{0,1,\dots,n\}$. It particularly holds that
$$\tc(X)= D(\mathrm{pr}_1,\mathrm{pr}_2), $$
where $\mathrm{pr}_i\colon X\times X \to X$ denotes the projection onto the $i$-th factor for $i \in \{1,2\}$. An analogue of Theorem \ref{TheoremTCNavigation} for homotopic distance was established in \cite[Theorem 4.5]{MVMLgeneralized}.

\begin{remark}
Morse-theoretic methods have been applied in topological robotics in other contexts as well, for example the in the study of configuration spaces of planar linkages by M. Farber and D. Sch\"utz, see \cite[Chapter 1]{FarberBook} or \cite[Section 3.1]{NicolaescuMorse} for an overview.
\end{remark}

\subsection{Critical points and numerical invariants} A reader familiar with critical point theory might see similarities between Theorem \ref{TheoremTCNavigation} and its proof and the Lusternik-Schnirelmann theorem, which we want to briefly recall. While Lusternik and Schni\-rel\-mann only considered finite-dimensional manifolds in their original works, the following is an extension of their work that is due to R. Palais, see \cite{PalaisLuster}.

\begin{theorem}[Lusternik-Schnirelmann]
\label{TheoremLS}
Let $M$ be a Hilbert manifold and let $f \in C^1(M)$ be bounded from below and satisfy the Palais-Smale condition with respect to a complete Riemannian metric on $M$. Then, for each $a \in \R$, 
$$\cat_M(f^a) \leq \sum_{c \in (-\infty,a]} (\cat_M(\mathrm{Crit}\, f \cap f_c)+1)-1.$$  
\end{theorem}
\begin{remark}
\begin{enumerate}[(1)]
	\item	In many statements of the Lusternik-Schnirelmann theorem, for example in \cite[Theorem 1.15]{CLOT} the function is required to be of class $C^2$ or of class $C^{1,1}$, i.e. to have a locally Lipschitz-continuous derivative. However, this requirement is dropped in Theorem \ref{TheoremLS} since every differentiable function on a smooth Banach manifold admits a locally Lipschitz-continuous \emph{pseudo}-gradient whose flow can be made use of in its proof, see \cite[Lemma II.3.9]{Struwe}.
	\item 
There are various generalizations of the Palais-Smale condition, under which analogues of Theorem \ref{TheoremLS} hold, e.g. those by M. Clapp and D. Puppe from \cite{ClappPuppe}. A more topological approach, which considers fixed points of self-maps instead of critical points, which are fixed points of time-1 map of a gradient flow, was carried out by Y. Rudyak and F. Schlenk in \cite{RudyakSchlenk}.	\end{enumerate}
\end{remark}
One derives from Theorem \ref{TheoremLS} that 
$$\# \{x \in \mathrm{Crit}\, f\ | \ f(x)\leq a\} \geq \cat_M(f^{ a})+1$$
for each $a \in \R$. For this reason, the Lusternik-Schnirelmann theorem is often used to estimate the minimum number of critical points of a function by topological quantities of the manifold. However, there are similar questions on functions whose critical points are not isolated, but occur in families to which the use of Lusternik-Schnirelmann category is limited. 

Consider a differentiable function $f:M \to \R$ on a Hilbert manifold which is invariant under the action of a Lie group $G$ on $M$. Then each critical point of $f$ is contained in a whole $G$-orbit of critical points of $M$. In particular, $f$ will have infinitely many critical points if $G$ is infinite. Evidently, in this case one is interested not in the number of critical points, but in the number of critical \emph{$G$-orbits} of $f$. \medskip

A particularly interesting example of such a function occurs in Riemannian geometry in connection with the \emph{closed geodesic problem}, which has been a frequent research topic in Riemannian geometry throughout the last half-century. In this problem, one asks if every closed Riemannian manifold admits infinitely many geometrically distinct non-constant closed geodesics, where we call two closed geodesics $\gamma_1,\gamma_2\colon S^1 \to M$ geometrically distinct if $\gamma_1(S^1) \neq \gamma_2(S^1)$. This question has been answered affirmatively for a large class of manifolds by D. Gromoll and W. Meyer in \cite{GromollMeyer} , namely for those manifolds, for which the sequence of Betti numbers of the free loop space is unbounded. Among those manifolds for whom the closed geodesic problem remains open are spheres and complex projective spaces. 

 Given a Riemannian manifold $(M,g)$, we recall that a closed geodesic with respect to $g$ is given by a geodesic segment $\alpha:[0,1]\to M$ which satisfies $\alpha(0)=\alpha(1)$ and $\alpha'(0)=\alpha'(1)$, which is more commonly expressed as a differentiable map $\gamma\colon S^1 \to M$. 
 For example, on $S^2$ with a round metric, the closed geodesics are precisely the great circles in $S^2$ with a suitable parametrization at constant speed. Most research on the closed geodesic problem roughly follows one of the following two approaches: either one studies all geodesics of the manifold by studying the geodesic flow on its tangent bundle and searches for closed orbits or one studies the free loop space of the manifold and searches for loops satisfying the closed geodesic condition. We refer the reader to \cite{Oancea} for a recent survey on the closed geodesics problem from the second point of view. The key tools in this approach are \emph{energy functionals}.

Instead of the free loop space $LM$ one studies the infinite-dimensional manifold $\Lambda M=H^1(S^1,M)$. It is locally modelled on the Sobolev space $H^1(S^1,\R^{\dim M})$ of weakly differentiable loops with square-integrable derivatives and has the homotopy type of $LM$, see \cite[Theorem 1.5.1]{MooreGlobal} for a proof. Since it is modelled on a Hilbert space, $\Lambda M$ is a Hilbert manifold. An advantage of Hilbert manifolds over general Banach manifolds is that several approaches and results from finite-dimensional Morse theory carry over to the infinite-dimensional setting, see e.g. \cite{PalaisHilbert} or \cite{AMHilbert}. The \emph{energy functional of $g$} is given by 
$$E_g\colon \Lambda M \to \R, \qquad E_g(\gamma) = \int_0^1 \|\gamma'(t)\|^2_{\gamma(t)}\, dt,$$
where for each $p \in M$ we denote by $\|\cdot\|_p$ the norm on $T_pM$ induced by $g$. The energy functional is indeed continuously differentiable and the critical points of $E_g$ are precisely the closed geodesics of $M$. It evidently holds that $E_g(\gamma) \geq 0$ for all $\gamma \in \Lambda M$ and that $E_g^{-1}(\{0\})= c(M)$, where $c(M)$ denotes the set of constant loops in $M$. \medskip

One shows that $E_g$ is invariant under the $O(2)$-action by reparametrizations which is induced by the action of $O(2)$ on $S^1$ via matrix multiplication. Consequently, the critical points of $E_g$, i.e. the closed geodesics of $g$, occur in $O(2)$-orbits. Moreover, since every constant loop minimizes $E_g$, each constant loop is a critical point of $E_g$ and we obtain $c(M)$ as a connected critical submanifold. Our aim is to find a lower bound for the number of $O(2)$-orbits of non-constant critical points of $E_g$. However, we can not gain information about this number from Theorem \ref{TheoremLS} as the following problems occur on the right-hand side of the inequality:
\begin{itemize}
	\item The value of $\cat_{\Lambda M}(c(M))$ is unclear. 
	\item Given $\gamma \in \Lambda M$, we know for its $O(2)$-orbit only that $\cat_{\Lambda M}(O(2)\cdot \gamma) \in \{0,1\}$, but there is no explicit way to decide which of the two values it takes. 
\end{itemize}
As mentioned above, similar problems occur more generally in Lusternik-Schnirelmann-type approaches to functions that are invariant under a Lie group action. There have been several approaches to handle such situations and to derive estimates on numbers of critical orbits, see \cite{Marz}, \cite{ClappPuppeSymm} and \cite{Bartsch}.

An idea suggested by the author in \cite{MescherSC} in the context of energy functional is now to replace the subspace categories $\cat_{\Lambda M}(A)$ by  better-suited numbers ${\sf {SC}}_{M}(A)$, for whom an inequality analogous to the one from Theorem \ref{TheoremLS} holds.

\subsection{Spherical complexities and critical points} We will formulate the definition and a main result in the more general context of sphere spaces, as no immediate technical benefit is obtained by restricting to loop spaces. Given $n \in \mathbb{N}_0$, we let $B^{n+1}\subset \R^{n+1}$ denote the closed $(n+1)$-dimensional unit ball and $S^n = \partial B^{n+1}$ denote the unit sphere in $\R^{n+1}$.
Given a path-connected topological space $X$, we put $B_{n+1}X := C^0(B^{n+1},X)$ and $S_nX:= \{f \in C^0(S^n,X) \ | \ f \text{ is nullhomotopic}\}$. Then 
$$r_n\colon B_{n+1} X \to S_nX, \qquad r_n(f) := f|_{S^n},$$
is a fibration.

\begin{definition}[{\cite{MescherSC}}]
\label{DefSC}
	Let $n \in \mathbb{N}$ and let $X$ be a path-connected topological space. 
Given $A \subset S_nX$ we let $i_A\colon A \hookrightarrow S_nX$ denote the inclusion and call
\begin{align*}
	{\sf {SC}}_{n,X}(A) := \secat(i_A^*r_n: i_A^*B_{n+1}X \to A)
\end{align*}
the \emph{spherical complexity of $A$ relative to $S_n X$}.
\end{definition}

\begin{remark}
\label{RemarkSC}
\begin{enumerate}[(1)]
	\item The case of $n=0$ in Definition \ref{DefSC} is already familiar to us. In this case, $B_1X \approx PX$ and $C^0(S^0,X) \approx X \times X$ and one checks that $r_0$ corresponds to the path fibration $p:PX \to X \times X$ under these identifications. Thus, for each $A \subset X \times X$, we obtain
	$${\sf {SC}}_{0,X}(A) = \tc_X(A).$$
	\item In the case of $n=1$, it apparently holds that $S_1X = L_0X$, the connected component of $LX$ consisting of all contractible loops in $X$ and the fibration $r_1:C^0(B^2,X) \to L_0X$ is  given by restricting a map from a $2$-disk to $M$ to its boundary loop. 
	\end{enumerate}
\end{remark}

The following result is an analogue of Theorem \ref{TheoremLS} for spherical complexities. Note that in the original formulation in \cite{MescherSC} the ``unreduced'' version of sectional category is used, so that the definition of the numbers ${\sf {SC}}_{n,X}(A)$ differs from the one in Definition \ref{DefSC} by one.

\begin{theorem}[{\cite[Corollary 2.14]{MescherSC}}]
\label{TheoremSC}
	Let $n \in \mathbb{N}_0$, let $X$ be a path-connected topological space and let $\mathcal{M} \subset S_nX$ be a smooth Hilbert manifold. Let $F:\mathcal{M} \to \R$ be continuously differentiable and bounded from below and assume that $F$ satisfies the Palais-Smale condition with respect to a complete Riemannian metric on $\mathcal{M}$.  Then 
$${\sf {SC}}_{n,X}(F^a) \leq \sum_{c \in (-\infty,a]} ({\sf {SC}}_{n,X}(\mathrm{Crit}\, F \cap F_c)+1)-1 \qquad \forall a \in \R. $$
\end{theorem}

In the light of Remark \ref{RemarkSC}.(1), let us focus on the case of $n=0$, i.e. of $\tc$. If we identify $S_0 X \approx X \times X$ and assume that $X$ is a smooth manifold, then we derive for $F:X \times X \to \R$ from Theorem \ref{TheoremSC} that 
$$\tc_M(F^a) \leq  \sum_{c \in (-\infty,a]} (\tc_M(\mathrm{Crit}\, F \cap F_c)+1)-1. $$
Comparing this to Theorem \ref{TheoremTCNavigation} shows that this indeed generalizes said theorem to  manifolds which are not necessarily closed. Moreover, if $F$ is non-negative and satisfies $F_0=\Delta_X$, then one obtains using Lemma \ref{LemmaGeod} that 
$$\tc_M(F^a) \leq  \sum_{c \in (0,a]} (\tc_M(\mathrm{Crit}\, F \cap F_c)+1), $$
which generalizes Theorem \ref{TheoremTCNavigation} in a straightforward way to functions which are not necessarily Morse-Bott, but satisfy the Palais-Smale condition, giving rise to a possible generalization of the class of navigation functions. The downside is that a detailed proof of this general inequality is technically more involved than the proof of Theorem \ref{TheoremTCNavigation} presented in this survey. \medskip 

For the rest of this subsection, we want to focus on the case of $n=1$, for which we assume that $X=M$ is a closed manifold and consider $\mathcal{M}=\Lambda_0M :=\{\alpha \in \Lambda M \ | \ \alpha \text{ is contractible}\}$. Since $\Lambda_0M$ is a connected component of $\Lambda M$, it is itself a Hilbert manifold. For all $A \subset \Lambda_0M$ we further put ${\sf {SC}}_M(A) := {\sf {SC}}_{1,M}(A)$; The following observation that is proven by elementary methods is key to the application of Theorem \ref{TheoremSC} to energy functionals. 

\begin{lemma}[{\cite[Propositions 1.4 and 3.1]{MescherSC}}]Let $M$ be a closed smooth manifold.
	\begin{enumerate}[a)]
		\item Let $c(M) \subset \Lambda_0M$ be the subset of constant loops. Then ${\sf {SC}}_M(c(M))=1$. 
		\item Let $O(2)\cdot \gamma$ be the orbit of a loop $\gamma$ under the $O(2)$-action by reparametrization. Then ${\sf {SC}}_M(O(2) \cdot \gamma)=1$ for each $\gamma \in \Lambda_0 M$. 
	\end{enumerate}
\end{lemma}

If we apply this lemma and Theorem \ref{TheoremSC} in the case of $n=1$ to the energy functional of Riemannian manifolds, we conclude that spherical complexities can indeed be used to \emph{count the number of $O(2)$-orbits in $\mathrm{Crit}\, E_g$} that are contained in a sublevel set.

\begin{cor}
\label{CorClosedGeod}
Let $(M,g)$ be a closed Riemannian manifold. Given $a \in \R$, let $N(g,a)$ denote the number of geometrically distinct contractible closed geodesics in $E_g^a$. Then
$$N(g,a) \geq {\sf {SC}}_{M}(E_g^a\cap \Lambda_0M).$$
\end{cor}

\begin{remark}
We note that analogues of Corollary \ref{CorClosedGeod} and of the following existence results for closed geodesics indeed hold true for \emph{Finsler} metrics. These are generalizations of Riemannian metrics that we shall not elaborate upon in this survey.
\end{remark}

We recall from the beginning of this subsection that the numbers ${\sf {SC}}_M(A)$ are obtained as sectional categories of pullback fibrations, which allows us to establish lower bounds for these numbers using the cohomology rings of the free loop space. This method is carried out in \cite[Section 7]{MescherSC} and in \cite{MescherExistGeod} by studying the Mayer-Vietoris sequence from subsection 1.2 and obtaining classes of sectional category weight two in the present context. We refer to \cite{MescherExistGeod} for the technical details and want to present a result that is derived using spherical complexities. 
  Given two Riemannian metrics $g_1$ and $g_2$ on the same manifold $M$ and $c>0$, we write $$g_1 \leq c \cdot g_2 \quad :\Leftrightarrow \quad  (g_1)_p(v,v) \leq c \cdot (g_2)_p(v,v) \qquad \forall p \in M, \ v \in T_pM.$$

\begin{theorem}[{\cite[Theorems 3.2 and 5.1]{MescherExistGeod}}] Let $(M,g)$ be a closed oriented Riemannian manifold. Let $K_g$ denote the sectional curvature of $g$ and let $\ell_g$ be the length of the shortest non-constant closed geodesic of $g$. If $0 < K_g \leq 1$ and if one of the following holds:
	\begin{enumerate}[(i))] 
	\item  $M=S^{2n}$ with $n \geq 2$ and $g \leq 4g_1$, where  $g_1$ denotes the round metric on $S^n$ of constant curvature $1$,
\item $M=\C P^n$ with $n \geq 3$ and $g \leq 4 g_1$, where $g_1$ denotes the Fubini-Study metric on $M$ of constant curvature $1$, 
	\end{enumerate}
	then $g$ admits two geometrically distinct closed geodesics of length less than $2 \ell_g$.
\end{theorem}

We refer to \cite{MescherExistGeod} for further existence results for closed geodesics that are obtained using spherical complexities, e.g. for odd-dimensional spheres, and to its introduction for an overview of other recent existence results.

\bibliography{Survey-Mescher.bib}
 \bibliographystyle{amsplain}

\end{document}